\documentclass[12pt]{amsart}

\usepackage{amssymb,latexsym}

\usepackage{enumerate}

\makeatletter

\@namedef{subjclassname@2010}{

  \textup{2010} Mathematics Subject Classification}

\makeatother
\newtheorem{thm}{Theorem}[section]
\newtheorem{Thm}{Theorem}

\newtheorem{lem}[thm]{Lemma}
\newtheorem{pro}[thm]{Proposition}
\theoremstyle{definition}

\numberwithin{equation}{section}

\newcommand{\vx}{{\bf x}}
\newcommand{\vt}{{\bf t}}
\newcommand{\vy}{{\bf y}}
\newcommand{\vs}{{\bf s}}
\newcommand{\pr}{\textup{P}}
\newcommand{\ex}{\mathbb{E}}
\newcommand{\Va}{\textup{Var}(q)}
\begin{document}

\baselineskip=17pt

\title[the limiting distribution in the Shanks–-R\'enyi
prime number race]{Large deviations of the limiting distribution in the Shanks–-R\'enyi
prime number race}

\author[Youness Lamzouri]{Youness Lamzouri}

\address{Department of Mathematics, University of Illinois at Urbana-Champaign,
1409 W. Green Street,
Urbana, IL, 61801
USA}

\email{lamzouri@math.uiuc.edu}

\date{}

\begin{abstract}
Let $q\geq 3$, $2\leq r\leq \phi(q)$ and $a_1,\dots,a_r$ be distinct residue classes modulo $q$ that are relatively prime to $q$. Assuming the Generalized Riemann Hypothesis GRH and the Linear Independence Hypothesis LI, M. Rubinstein and P. Sarnak \cite{RS} showed that the vector-valued function $E_{q;a_1,\dots,a_r}(x)=(E(x;q,a_1), \dots, E(x;q,a_r)),$ where $E(x;q,a)= \frac{\log x}{\sqrt{x}}\left(\phi(q)\pi(x;q,a)-\pi(x)\right)$, has a limiting distribution $\mu_{q;a_1,\dots,a_r}$ which is absolutely continuous on $\mathbb{R}^r$. Furthermore, they proved that for $r$ fixed, $\mu_{q;a_1,\dots,a_r}$ tends to a multidimensional Gaussian as $q\to\infty$. In the present paper, we determine the exact rate of this convergence, and investigate the asymptotic behavior of the large deviations of $\mu_{q;a_1,\dots,a_r}$.
\end{abstract}



\thanks{The author is supported by a Postdoctoral Fellowship from the Natural Sciences and Engineering Research Council of Canada.}

\maketitle

\section{Introduction}
A classical problem in analytic number theory is the so-called `` Shanks and R\'enyi prime number race'' (see \cite{KT}) which is described in the following way. Let $q\geq 3$ and $2\leq r\leq \phi(q)$ be positive integers. For an ordered $r$-tuple of distinct reduced residues $(a_1,a_2,\dots, a_r)$ modulo $q$ we denote by $P_{q; a_1,\dots, a_r}$ the set of real numbers $x\geq 2$ such that
$$ \pi(x;q,a_1)>\pi(x;q,a_2)>\dots>\pi(x;q,a_r).$$
Will the sets $P_{q; a_{\sigma(1)},\dots, a_{\sigma(r)}}$ contain arbitrarily large values, for any permutation $\sigma$ of $\{1,2,\dots, r\}$?

A result of J. E. Littlewood \cite{Li} from 1914 shows  that the answer is yes in the cases $(q,a_1,a_2)=(4,1,3)$ and $(q,a_1,a_2)=(3,1,2)$.
Similar results to other moduli in the case $r=2$ were subsequently derived by S. Knapowski and P. Tur\'an \cite{KT} (under some hypotheses on the zeros of Dirichlet $L$-functions), and special cases of the prime number race with $r\geq 3$  were considered by J. Kaczorowski \cite{Ka1}, \cite{Ka2}.

In 1994,  M. Rubinstein and P. Sarnak \cite{RS} completely solved this problem, conditionally on the assumptions of the Generalized Riemann Hypothesis GRH and
the Linear Independence Hypothesis LI (which is the assumption that the nonnegative imaginary parts of the zeros of all Dirichlet $L$-functions attached to primitive characters modulo $q$ are linearly independent over $\mathbb{Q}$). To describe their results, we first define some notation.
For any real number $x\geq 2$ we introduce the vector-valued function
$$ E_{q;a_1,\dots,a_r}(x):=(E(x;q,a_1), \dots, E(x;q,a_r)),$$
where
$$ E(x;q,a):=\frac{\log x}{\sqrt{x}}\left(\phi(q)\pi(x;q,a)-\pi(x)\right).$$
The normalization is such that, if we assume GRH, $E_{q;a_1,\dots,a_r}(x)$ varies roughly boundedly as $x$ varies.
Rubinstein and Sarnak showed, assuming GRH, that the function $E_{q;a_1,\dots,a_r}(x)$ has a limiting distribution $\mu_{q;a_1,\dots,a_r}$. More precisely, they proved
\begin{equation}
 \lim_{X\to\infty}\frac{1}{\log X} \int_2^X f\left(E_{q;a_1,\dots,a_r}(x)\right)\frac{dx}{x}= \int_{\mathbb{R}^r}f(x_1,\dots,x_r)
d\mu_{q;a_1,\dots,a_r},
\end{equation}
for all bounded, continuous functions $f$ on $\mathbb{R}^r$. Furthermore, assuming both GRH and LI, they showed that $\mu_{q;a_1,\dots,a_r}$ is absolutely continuous with respect to the Lebesgue measure on $\mathbb{R}^r$ if $r<\phi(q)$. (When $r=\phi(q)$, $\mu_{q;a_1,\dots,a_r}$ is shown to be absolutely continuous with respect to the Lebesgue measure on the hyperplane $\sum_{j=1}^rx_j=0$.)  As a consequence, under GRH and LI, the logarithmic density of the set $P_{q;a_1,\dots,a_r}$ defined by
$$ \lim_{x\to \infty}\frac{1}{\log x}\int_{t\in P_{q;a_1,\dots, a_r}\cap [2,x]}\frac{dt}{t}$$
exists and is positive.

Here and throughout we shall use the notations $||\vx||=\sqrt{\sum_{j=1}^r x_j^2}$ and $|\vx|_{\infty}=\max_{1\leq i\leq r}|x_i|$ for the Euclidean norm and the maximum norm on $\mathbb{R}^r$ respectively. In \cite{RS}, Rubinstein and Sarnak also studied the behavior of the tail $\mu_{q;a_1,\dots,a_r}(||\vx||>V)$ when $q$ is fixed. They showed, under GRH, that for all distinct reduced residues $a_1,\dots,a_r$ modulo $q$ we have
\begin{equation}
 \exp\left(-\exp(c_2(q)V)\right)\ll_q \mu_{q;a_1,\dots,a_r}(||\vx||>V) \ll_q\exp\left(-c_1(q)\sqrt{V}\right),
\end{equation}
for some $c_1(q), c_2(q)>0$ which depend on $q$.

In this paper we investigate large deviations of the distribution $\mu_{q;a_1,\dots,a_r}$ uniformly as $q\to\infty$, under the additional assumption of LI. For a non-trivial character $\chi$ modulo $q$, we denote by $\{\gamma_{\chi}\}$ the sequence of imaginary parts of the non-trivial zeros of $L(s,\chi)$. Let $\chi_0$ denote the principal character modulo $q$ and define $S= \cup_{\chi\neq \chi_0\bmod q}\{\gamma_{\chi}\}$. Moreover, let  $\{U(\gamma_{\chi})\}_{\gamma_{\chi}\in S}$ be a sequence of independent random variables uniformly distributed on the unit circle. Rubinstein and Sarnak established, under GRH and LI, that the distribution $\mu_{q;a_1,\dots,a_r}$ is the same as the probability measure corresponding to the random vector
\begin{equation}
X_{q;a_1,\dots,a_r}= (X(q,a_1),\dots,X(q,a_r)),
\end{equation}
where
$$ X(q,a)= -C_q(a)+ \sum_{\substack{\chi\neq \chi_0\\ \chi\bmod q}}\sum_{\gamma_{\chi}>0} \frac{2\text{Re}(\chi(a)U(\gamma_{\chi}))}{\sqrt{\frac14+\gamma_{\chi}^2}},$$
$\chi_0$ is the principal character modulo $q$ and $$C_q(a):=-1+ \sum_{\substack{b^2\equiv a \bmod q\\ 1\leq b\leq q}}1.$$
Note that for $(a,q)=1$ the function $C_q(a)$ takes only two values: $C_q(a)=-1$ if $a$ is a non-square modulo $q$, and $C_q(a)=C_q(1)$ if $a$ is a square modulo $q$. Furthermore, an elementary argument shows that $C_q(a)<d(q)\ll_{\epsilon}q^{\epsilon}$ for any $\epsilon>0$, where $d(q)=\sum_{m|q}1$ is the usual divisor function.

Let $\text{Cov}_{q;a_1,\dots,a_r}$ be the covariance matrix of $X_{q;a_1,\dots,a_r}$. A straightforward computation shows that the entries of  $\text{Cov}_{q;a_1,\dots,a_r}$ are

$$\textup{Cov}_{q;a_1,\dots,a_r}(j,k)= \begin{cases}  \Va &\text{ if } j=k \\ B_q(a_j,a_k) &\text{ if } j\neq k,\end{cases}$$
where
$$
\Va:=\sum_{\substack{\chi\neq \chi_0\\ \chi\bmod q}}\sum_{\gamma_{\chi}}\frac{1}{\frac14+\gamma_{\chi}^2}=2\sum_{\substack{\chi\neq \chi_0\\ \chi\bmod q}}\sum_{\gamma_{\chi}>0}\frac{1}{\frac14+\gamma_{\chi}^2},$$
  and $$B_q(a,b):=\sum_{\substack{\chi\neq \chi_0 \\ \chi\bmod q}}\sum_{\gamma_{\chi}>0}\frac{\chi\left(\frac{b}{a}\right)+\chi\left(\frac{a}{b}\right)}{\frac14 +\gamma_{\chi}^2}
$$
for $(a,b)\in \mathcal{A}(q)$, where $\mathcal{A}(q)$ is the set of ordered pairs of distinct reduced residues modulo $q$. Assuming GRH, Rubinstein and Sarnak showed that
\begin{equation}
\Va\sim \phi(q)\log q \text{ and } B_q(a,b)=o(\phi(q)\log q) \text{ as } q\to\infty,
\end{equation}
uniformly for all $(a,b)\in \mathcal{A}(q)$. Combining these estimates with an explicit formula for the Fourier transform of $\mu_{q;a_1,\dots,a_r}$ (see (3.1) below) they established, under GRH and LI, that
\begin{equation}
\mu_{q;a_1,\dots,a_r}\left(||\vx||>\lambda\sqrt{\Va}\right)= (2\pi)^{-r/2}\int_{||\vx||>\lambda} \exp\left(-\frac{x_1^2+\cdots+x_r^2}{2}\right)d\vx +o_q(1),
\end{equation}
for any fixed $\lambda>0$.

We refine their result using the approach developed in \cite{La}. More precisely, we prove that the asymptotic formula (1.5) holds uniformly in the range $0<\lambda\leq \sqrt{\log\log q}$ with an optimal error term $O_r(1/\log^2q).$

\begin{Thm} Assume GRH and LI. Fix an integer $r\geq 2$. Let $q$ be large and $a_1,\dots,a_r$ be distinct reduced residues modulo $q$. Then, in the range $0< \lambda\leq \sqrt{\log\log q}$ we have
$$
 \mu_{q;a_1,\dots,a_r}\left(||\vx||>\lambda\sqrt{\Va}\right)=
(2\pi)^{-r/2}\int_{||\vx||>\lambda} \exp\left(-\frac{x_1^2+\cdots+x_r^2}{2}\right)d\vx
 + O_r\left(\frac{1}{\log^2q}\right).
$$
Moreover, there exists an $r$-tuple of distinct reduced residue classes $(a_1,\dots,a_r)$ modulo $q$, such that in the range $1/4<\lambda<3/4$ we have
$$
 \left|\mu_{q;a_1,\dots,a_r}\left(||\vx||>\lambda\sqrt{\Va}\right)-
(2\pi)^{-r/2}\int_{||\vx||>\lambda} \exp\left(-\frac{x_1^2+\cdots+x_r^2}{2}\right)d\vx\right|
\gg_r\frac{1}{\log^2q}.
$$
\end{Thm}
Since $\Va\sim\phi(q)\log q$, it follows from Theorem 1 that
\begin{equation}
 \mu_{q;a_1,\dots,a_r}(||\vx||>V)=\exp\left(-\frac{V^2}{2\phi(q)\log q}(1+o(1))\right)
\end{equation}
in the range $(\phi(q)\log q)^{1/2} \ll V\leq (1+o(1))(\phi(q)\log q\log\log q)^{1/2}.$ Exploiting the results of H. L. Montgomery and A. M. Odlyzko \cite{MO}, we prove that a similar behavior holds in the much larger range $(\phi(q)\log q)^{1/2}\ll V\ll \phi(q)\log q.$

\begin{Thm} Assume GRH and LI. Fix an integer $r\geq 2$ and a real number $A\geq 1$. Let $q$ be large. Then for all distinct reduced residues $a_1,\dots,a_r$ modulo $q$, we have
$$ \exp\left(-c_2(r,A) \frac{V^2}{\phi(q)\log q}\right) \ll  \mu_{q;a_1,\dots,a_r}(||\vx||>V) \ll \exp\left(-c_1(r,A) \frac{V^2}{\phi(q)\log q}\right),$$
uniformly in the range $(\phi(q)\log q)^{1/2}\ll V\leq A\phi(q)\log q$, where $c_2(r,A)>c_1(r,A)$ are positive numbers that depend only on $r$ and $A$.
\end{Thm}
Using an analogous approach, we prove that (1.6) does not hold when $V/(\phi(q)\log q)\to\infty$ as $q\to\infty$, which shows that a transition occurs at $V\asymp \phi(q)\log q$.

\begin{Thm} Assume GRH and LI. Fix an integer $r\geq 2$ and let $q$ be large. If $V/(\phi(q)\log q)\to\infty$ and $V/(\phi(q)\log^2 q)\to 0$ as $q\to\infty$, then for all distinct reduced residues $a_1,\dots,a_r$ modulo $q$, we have
$$ \exp\left(-c_4(r) \frac{V^2}{\phi(q)\log q }\exp\left(c_6(r)\frac{V}{\phi(q)\log q}\right)\right) \ll  \mu_{q;a_1,\dots,a_r}(||\vx||>V)$$
and
$$ \mu_{q;a_1,\dots,a_r}(||\vx||>V) \ll \exp\left(-c_3(r) \frac{V^2}{\phi(q)\log q}\exp\left(c_5(r)\frac{V}{\phi(q)\log q}\right)\right),$$
where $c_4(r)>c_3(r)$, and $c_6(r)>c_5(r)$ are positive numbers which depend only on $r$.
\end{Thm}
In the range $V/(\phi(q)\log^2q)\to\infty$ one can prove, using the same ideas, that there are positive constants $c_8(r)>c_7(r)$ such that
\begin{equation}
\exp\left(-\exp\left(c_8(r)\sqrt{\frac{V}{\phi(q)}}\right)\right) \leq\mu_{q;a_1,\dots,a_r}(||\vx||>V)\leq \exp\left(-\exp\left(c_7(r)\sqrt{\frac{V}{\phi(q)}}\right)\right).
\end{equation}
In particular, these bounds show that the asymptotic behavior of the tail $\mu_{q;a_1,\dots,a_r}(||\vx||>V)$ changes again at $V\asymp \phi(q)\log^2 q.$
Assuming the RH and using the LI for the Riemann zeta function, H. L. Montgomery \cite{Mont} had previously obtained a similar result for $\mu_1$, the limiting distribution of the error term in the prime number theorem  $\pi(x)-\text{Li}(x)$. His result states that
$$ \exp\left(-c_{10}\sqrt{V}\exp\left(\sqrt{2\pi V}\right)\right)\leq \mu_1(|x|>V)\leq \exp\left(-c_9\sqrt{V}\exp\left(\sqrt{2\pi V}\right)\right),$$ for some absolute constants $c_{10}>c_{9}>0$. A more precise estimate was subsequently derived by W. Monach \cite{Mona}, namely
 \begin{equation}
  \mu_1(|x|>V)= \exp\left(-(e^{-A_0}+o(1))\sqrt{2\pi V}\exp\left(\sqrt{2\pi V}\right)\right),
 \end{equation}
 where $A_0$ is an absolute constant defined in Theorem 4 below.

In our case, it appears that $\mu_{q;a_1,\dots, a_r}(|\vx|_{\infty}>V)$ is more natural to study, if one wants to gain a better understanding of the decay rate of large deviations of $\mu_{q;a_1,\dots,a_r}$ in the range $V/(\phi(q)\log^2 q)\to\infty$. We achieved this using the saddle-point method. We also note that in contrast to our previous results, $r$ can vary uniformly in $[2,\phi(q)-1]$ as $q\to\infty$.
\begin{Thm} Assume GRH and LI. Let $q$ be large, and $2\leq r\leq \phi(q)-1$ be an integer. If $V/(\phi(q)\log^2 q)\to\infty$ as $q\to\infty$, then for all distinct reduced residue classes $a_1,\dots,a_r$ modulo $q$, the tail $\mu_{q;a_1,\dots, a_r}(|\vx|_{\infty}>V)$ equals
$$
\exp\left(-e^{-L(q)}\sqrt{\frac{2(\phi(q)-1) V}{\pi}}\exp\left(\sqrt{L(q)^2+\frac{2\pi V}{\phi(q)-1}}\right)\left(1+ O\left(\left(\frac{\phi(q)\log^2q}{V}\right)^{1/4}\right)\right)\right),$$
where $$L(q)=\frac{\phi(q)}{\phi(q)-1}\left(\log q -\sum_{p|q}\frac{\log p}{p-1}\right)+A_0-\log\pi,$$
and
$$A_0:=\int_{0}^{1}\frac{\log I_0(t)}{t^2}dt+\int_{1}^{\infty}\frac{\log I_0(t)-t}{t^2}dt+1=1,2977474,$$
where $I_0(t)=\sum_{n=0}^{\infty} (t/2)^{2n}/n!^2$ is the modified Bessel function of order $0$.
\end{Thm}
Lastly, we should also mention that in the range $V\geq \sqrt{\phi(q)}\log q$ one may allow $r$ to vary uniformly, as in Theorem 4, if one is willing to replace $\mu_{q;a_1,\dots, a_r}(||\vx||>V)$ by $\mu_{q;a_1,\dots, a_r}(|\vx|_{\infty}>V)$ in the statements of Theorems 2 and 3.

\section{The intermediate range: Proof of Theorems 2 and 3}

In this section we investigate the behavior of the tail $\mu_{q;a_1,\dots,a_r}(||x||>V)$ when $V\gg \sqrt{\phi(q)\log q}$ as long as $V/(\phi(q)\log^2q)\to 0$ as $q\to\infty$. Recall that $\mu_{q;a_1,\dots,a_r}$ is also the probability measure corresponding to the random vector $X_{q;a_1,\dots,a_r}$ (defined in (1.3)). Our idea starts with the observation that the random variables
\begin{equation} Y(q,a):=X(q,a)-\ex(X(q,a))= \sum_{\substack{\chi\neq \chi_0\\ \chi\bmod q}}\sum_{\gamma_{\chi}>0} \frac{2\text{Re}(\chi(a)U(\gamma_{\chi}))}{\sqrt{\frac14+\gamma_{\chi}^2}},
\end{equation}
are identically distributed for all reduced residues $a$ modulo $q$. Indeed, for all $(a,q)=1$ the random variables $\{\tilde{U}(\gamma_{\chi})\}_{\gamma_{\chi}\in S}$, where $\tilde{U}(\gamma_{\chi})= \chi(a)U(\gamma_{\chi})$, are independent and uniformly distributed on the unit circle. Hence, if $(a,q)=1$ then  $Y(q,a)$ has the same distribution as
\begin{equation}
Y(q):= \sum_{\substack{\chi\neq \chi_0\\ \chi\bmod q}}\sum_{\gamma_{\chi}>0} \frac{2\cos(2\pi\theta(\gamma_{\chi}))}{\sqrt{\frac14+\gamma_{\chi}^2}},
\end{equation}
where $\{\theta(\gamma_{\chi})\}_{\gamma_{\chi}\in S}$ are independent random variables uniformly distributed on $[0,1]$. Our first lemma shows, in our range of $V$, that large deviations of $\mu_{q;a_1,\dots,a_r}$ are closely related to those of $Y(q)$.
\begin{lem} The random variable $Y(q)$ is symmetric. Moreover, if $q$ is sufficiently large, $r\geq 2$ is fixed and $V\geq \sqrt{\phi(q)}$, then for all distinct reduced residues $a_1,\dots,a_r$ modulo $q$ we have
$$
 \pr(Y(q)> 2V)\leq \mu_{q;a_1,\dots,a_r}(||x||>V) \leq 2r \pr\left(Y(q)> \frac{V}{2\sqrt{r}}\right).
$$
\end{lem}
\begin{proof} Note that $\ex(\exp(it\cos(2\pi\theta(\gamma_{\chi}))))=J_0(t)$, where $J_0(t):=\sum_{m=0}^{\infty}(-1)^{m}(t/2)^{2m}/m!^2$ is the Bessel function of order $0$.  Therefore, since the Fourier transform (characteristic function) of $Y(q)$ is
$$ \ex\left(e^{itY(q)}\right)= \prod_{\substack{\chi\neq \chi_0\\ \chi\bmod q}}\prod_{\gamma_{\chi}>0}J_0\left(\frac{2t}{{\sqrt{\frac14+\gamma_{\chi}^2}}}\right),$$
 and $J_0$ is an even function then $Y(q)$ is symmetric.  Now,
$$ \mu_{q;a_1,\dots,a_r}(||x||>V)=\pr(||X_{q;a_1,\dots,a_r}||>V)\geq \pr(X(q,a_1)>V)= \pr(Y(q)> V+C_q(a_1)).$$
The lower bound follows from the fact that $C_q(a_1)< d(q)=q^{o(1)}.$ On the other hand, if $|X(q,a_j)|\leq V/\sqrt{r}$ for all $1\leq j\leq r$, then $||X_{q;a_1,\dots,a_r}||\leq V$. Hence
$$ \mu_{q;a_1,\dots,a_r}(||x||>V)\leq \sum_{j=1}^r\pr\left(|X(q,a_j)|> \frac{V}{\sqrt{r}}\right)\leq \sum_{j=1}^r\pr\left(|Y(q)|> \frac{V}{\sqrt{r}}-|C_q(a_j)|\right).$$
Using that $V\geq \sqrt{\phi(q)}$, $|C_q(a)|=q^{o(1)}$ for all $(a,q)=1$, and $Y(q)$ is symmetric we obtain from the last inequality
$$ \mu_{q;a_1,\dots,a_r}(||x||>V)\leq r\pr\left(|Y(q)|>\frac{V}{2\sqrt{r}}\right)= 2r\pr\left(Y(q)>\frac{V}{2\sqrt{r}}\right), $$
if $q$ is sufficiently large. This establishes the lemma.

\end{proof}
To investigate large deviations of $Y(q)$, we shall appeal to the following result of H. L. Montgomery and A. M. Odlyzko \cite{MO}.
\begin{thm}[Theorem 2 of \cite{MO}] Let $\{Y_n\}_{n\geq 1}$ be a sequence of independent real valued random variables such that $\ex(Y_n)=0$ and $|Y_n|\leq 1$. Suppose there is constant $c>0$ such that $\ex(Y_n^2)\geq c$ for all $n$. Put $Y=\sum_{n\geq 1} Y_n/r_n$ where $\sum_{n\geq 1}1/r_n^2<\infty.$
If $\sum_{|r_n|\leq T}|r_n|^{-1}\leq V/2$, then
\begin{equation}
\pr(Y\geq V)\leq \exp\left(-\frac{1}{16}V^2\left(\sum_{|r_n|>T}r_n^{-2}\right)^{-1}\right).
\end{equation}
Moreover, if $\sum_{|r_n|\leq T}|r_n|^{-1}\geq  2V$, then
\begin{equation}
\pr(Y\geq V)\geq a_1\exp\left(-a_2V^2\left(\sum_{|r_n|>T}r_n^{-2}\right)^{-1}\right),
\end{equation}
where $a_1, a_2>0$ depend only on $c$.
\end{thm}
In order to apply this result to our setting, we have to understand the asymptotic behavior of the sums $\sum_{\chi\neq\chi_0}\sum_{0<\gamma_{\chi}\leq T} 1/(\frac{1}{4}+\gamma_{\chi}^2)^{1/2}$ and $\sum_{\chi\neq\chi_0}\sum_{\gamma_{\chi}> T}1/(\frac{1}{4}+\gamma_{\chi}^2)$.
For a non-trivial character $\chi$ modulo $q$, we let $q^*_{\chi}$ be the conductor of $\chi$, and $\chi^*$ be the unique primitive character modulo $q^*_{\chi}$ which induces $\chi$. We begin by recording some standard estimates which will be useful in our subsequent work.
\begin{lem} Assume GRH. Let $\chi$ be a non-trivial character modulo $q$. Then
$$
\sum_{\gamma_{\chi}>0}\frac{1}{\frac14+ \gamma_{\chi}^2}
= \frac12\log q_{\chi}^*+ O(\log\log q).
$$
Moreover, we have
$$
 \sum_{\chi \bmod q}\log q_{\chi}^*=\phi(q)\left(\log q-\sum_{p|q}\frac{\log p}{p-1}\right),
$$
and
$$
\sum_{p|q}\frac{\log p}{p-1}\ll \log\log q.
$$
\end{lem}
\begin{proof}
The first estimate follows from Lemma 3.5 of \cite{FiM}, and the second is proved in Proposition 3.3 of \cite{FiM}. Finally, we have
$$\sum_{p|q}\frac{\log p}{p-1}\leq \sum_{p\leq (\log q)^2}\frac{\log p}{p-1}+\frac{1}{\log q}\sum_{p|q}1\ll \log\log q,$$
which follows from the trivial bound $\sum_{p|q}1 \leq \log q/\log2$. 
\end{proof}
From this lemma, one can deduce the more precise asymptotic
$$\Va=\phi(q)\log q+O(\phi(q)\log\log q).$$ Our next result gives the classical estimate for
$$N_q(T):=\sum_{\substack{\chi\neq \chi_0\\ \chi\bmod q}}\sum_{0<\gamma_{\chi}<T} 1.$$
\begin{lem}
If $T\geq 2$, then
$$
 N_q(T)=\frac{(\phi(q)-1)}{2\pi}T\log T+\frac{R(q)}{2\pi}T
+O(\phi(q)\log qT),
$$
where
$$ R(q):=\phi(q)\left(\log q-\sum_{p|q}\frac{\log p}{p-1}\right)-(\phi(q)-1)(\log 2\pi+1).$$
\end{lem}
\begin{proof}
Using Lemma 2.3 and appealing to the classical estimates for the number of zeros of Dirichlet $L$-functions (see Chapters 15 and 16 of \cite{Da}), we get
\begin{equation*}
\begin{aligned}
 N_q(T)&=\sum_{\substack{\chi\neq \chi_0\\ \chi\bmod q}}\left(\frac{T}{2\pi}\log \frac{q^*_{\chi}T}{2\pi}-\frac{T}{2\pi}+O(\log q T)\right)\\
&=\frac{(\phi(q)-1)}{2\pi}T\log T+\frac{R(q)}{2\pi}T
+O(\phi(q)\log qT),
\end{aligned}
\end{equation*}
as desired.
\end{proof}
Using Lemmas 2.3 and 2.4 we establish:
\begin{pro} There exists a constant $T_0\geq 2$ such that if $T\geq T_0$, and $\log T/\log q\to 0$ as $q\to \infty$, then
$$ \frac{1}{20}\phi(q)\log q\log T\leq \sum_{\substack{\chi\neq \chi_0\\ \chi\bmod q}}\sum_{0<\gamma_{\chi}\leq T} \frac{1}{\sqrt{\frac{1}{4}+\gamma_{\chi}^2}} \leq 3\phi(q)\log q\log T$$
if $q$ is sufficiently large, and
$$ \sum_{\substack{\chi\neq \chi_0\\ \chi\bmod q}}\sum_{\gamma_{\chi}> T} \frac{1}{\frac{1}{4}+\gamma_{\chi}^2}= \frac{\phi(q)\log q}{2\pi T}\left(1+ O\left(\frac{\log T}{\log q}+ \frac{\log\log q}{\log q}+\frac{1}{T}\right)\right).$$
\end{pro}
\begin{proof}
First, since $\log T=o(\log q)$ then Lemmas 2.3 and 2.4 yield
\begin{equation}
\begin{aligned}
N_q(T)&= \frac{T}{2\pi}\phi(q)\log q +O\left(\phi(q)T\log T+ T\phi(q)\log\log q+ \phi(q)\log q\right)\\
&= \frac{T}{2\pi}\phi(q)\log q\left(1+O\left(\frac{\log T}{\log q} + \frac{\log\log q}{\log q} + \frac{1}{T}\right)\right).
\end{aligned}
\end{equation}
Hence, there exists a suitably large constant $T_1\geq 2$ such that
\begin{equation}
\frac{T}{10}\phi(q)\log q\leq  N_q(T)\leq T\phi(q)\log q,
\end{equation}
if $T\geq T_1$ and $\log T=o(\log q)$. Now assume that $T\geq T_1^2$. Then
$$ \sum_{\substack{\chi\neq \chi_0\\ \chi\bmod q}}\sum_{0<\gamma_{\chi}\leq T} \frac{1}{\sqrt{\frac{1}{4}+\gamma_{\chi}^2}}= \int_{0}^T\frac{dN_q(t)}{\sqrt{\frac14+t^2}}= \frac{N_q(T)}{\sqrt{\frac14+T^2}}+ 2\int_{0}^T\frac{tN_q(t)}{\left(\frac14+t^2\right)^{3/2}}dt.$$
Then, using the lower bound of (2.6) we get
$$\sum_{\substack{\chi\neq \chi_0\\ \chi\bmod q}}\sum_{0<\gamma_{\chi}\leq  T} \frac{1}{\sqrt{\frac{1}{4}+\gamma_{\chi}^2}} \geq \frac{\phi(q)\log q}{10}\int_{T_1}^T\frac{2t^2}{\left(\frac14+t^2\right)^{3/2}}dt\geq \frac{1}{20}\phi(q)\log q\log T,$$
which follows from the fact that $2t^3\geq (1/4+t^2)^{3/2}$ for $t\geq 1$.
Similarly, we obtain from the upper bound of (2.6)
$$ \int_{0}^{T_1}\frac{tN_q(t)}{\left(\frac14+t^2\right)^{3/2}}dt\ll_{T_1} \phi(q)\log q \text{ and } \int_{T_1}^{T}\frac{tN_q(t)}{\left(\frac14+t^2\right)^{3/2}}dt \leq \phi(q)\log q\log T.$$
This implies
$$\sum_{\substack{\chi\neq \chi_0\\ \chi\bmod q}}\sum_{0<\gamma_{\chi}\leq  T} \frac{1}{\sqrt{\frac{1}{4}+\gamma_{\chi}^2}}\leq 2\phi(q)\log q\log T+ O(\phi(q)\log q).$$
Therefore, if $T\geq T_0$ for some suitably large constant $T_0$, then
$$
\frac{1}{20}\phi(q)\log q\log T\leq \sum_{\substack{\chi\neq \chi_0\\ \chi\bmod q}}\sum_{0<\gamma_{\chi}\leq  T}\frac{1}{\sqrt{\frac{1}{4}+\gamma_{\chi}^2}}\leq 3\phi(q)\log q\log T.
$$
On the other hand, we have
$$ \sum_{\substack{\chi\neq \chi_0\\ \chi\bmod q}}\sum_{\gamma_{\chi}>T} \frac{1}{\frac{1}{4}+\gamma_{\chi}^2}= \int_{T}^{\infty}\frac{dN_q(t)}{\frac14+t^2}= -\frac{N_q(T)}{\frac14+T^2}+ 2\int_{T}^{\infty}\frac{tN_q(t)}{\left(\frac14+t^2\right)^{2}}dt.$$
Thus, inserting the estimate (2.5) into the previous identity gives
\begin{equation}
\sum_{\substack{\chi\neq \chi_0\\ \chi\bmod q}}\sum_{\gamma_{\chi}>T} \frac{1}{\frac{1}{4}+\gamma_{\chi}^2}= - \frac{T\phi(q)\log q}{2\pi(\frac14+T^2)} + \frac{\phi(q)\log q}{\pi}\int_{T}^{\infty}\frac{t^2}{\left(\frac14+t^2\right)^{2}}dt + E_1,
\end{equation}
where $$E_1\ll \frac{\phi(q)\log q}{T}\left(\frac{\log T}{\log q}+ \frac{\log\log q}{\log q}+\frac{1}{T}\right).$$
Moreover, the main term of (2.7) equals
$$\frac{\phi(q)\log q}{2\pi T}\left(1+O\left(\frac{1}{T}\right)\right),$$
as desired.
\end{proof}
\begin{proof}[Proof of Theorem 2]
Since $\ex(\cos(2\pi \theta(\gamma_{\chi})))=0$ and $\ex(\cos^2(2\pi \theta(\gamma_{\chi})))=1/2>0$, then we can apply Theorem 2.2 to derive upper and lower bounds for $\pr(Y(q)\geq V).$  We first establish the upper bound. Taking $T=0$ in (2.3) yields
\begin{equation}
\pr(Y(q)\geq V)\leq \exp\left(-\frac{V^2}{32\Va}\right)\leq \exp\left(-\frac{V^2}{40\phi(q)\log q}\right),
\end{equation}
if $q$ is sufficiently large, since $\Va\sim \phi(q)\log q.$ Now, let $T=T(A)$ be a suitably large number such that $\log T\geq 40 A$. Then Proposition 2.5 implies
$$ \sum_{\substack{\chi\neq \chi_0\\ \chi\bmod q}}\sum_{\substack{(1/4+\gamma_{\chi}^2)^{1/2}\leq 2T\\ \gamma_{\chi}>0}}\frac{1}{\sqrt{\frac{1}{4}+\gamma_{\chi}^2}}\geq\sum_{\substack{\chi\neq \chi_0\\ \chi\bmod q}}\sum_{0<\gamma_{\chi}\leq T}\frac{1}{\sqrt{\frac{1}{4}+\gamma_{\chi}^2}}\geq \frac{1}{20}\phi(q)\log q\log T\geq 2V,$$
since $V\leq A\phi(q)\log q$. On the other hand, using Proposition 2.5 we get
$$ \sum_{\substack{\chi\neq \chi_0\\ \chi\bmod q}}\sum_{\substack{(1/4+\gamma_{\chi}^2)^{1/2}>2T\\ \gamma_{\chi}>0}}\frac{1}{\frac{1}{4}+\gamma_{\chi}^2}\geq \sum_{\substack{\chi\neq \chi_0\\ \chi\bmod q}}\sum_{\gamma_{\chi}>2T}\frac{1}{\frac{1}{4}+\gamma_{\chi}^2}\geq \frac{\phi(q)\log q}{8\pi T},$$
if $T$ is suitably large.
Hence, applying (2.4) we derive
\begin{equation}
\pr(Y(q)\geq V)\gg \exp\left(-c_1(A)\frac{V^2}{\phi(q)\log q}\right),
\end{equation}
for some positive number $c_1(A)$ which depends only on $A$. The theorem follows upon combining the bounds (2.8) and (2.9) with Lemma 2.1.
\end{proof}
\begin{proof}[Proof of Theorem 3]
The result can be deduced by proceeding along the same lines as in the proof of Theorem 2. Indeed, using Lemma 2.1, Theorem 2.2 and Proposition 2.5, the lower bound in Theorem 3 can be derived by taking $T= \exp(50 V/(\phi(q)\log q))$. Similarly, to get the corresponding upper bound choose $T= \exp(V/(10\phi(q)\log q)).$
\end{proof}

\section{Approximating $\mu_{q;a_1,\dots,a_r}$ by a multivariate Gaussian distribution: Proof of Theorem 1}

Assuming GRH and LI, Rubinstein and Sarnak obtained an explicit formula for the Fourier transform of $\mu_{q;a_1,\dots,a_r}$ in terms of the non-trivial zeros of Dirichlet $L$-functions attached to non-principal characters modulo $q$. More specifically they showed that
\begin{equation}
\hat{\mu}_{q;a_1,\dots, a_r}(t_1,\dots,t_r)=  \exp\left(i\sum_{j=1}^rC_q(a_j)t_j\right)\prod_{\substack{\chi\neq \chi_0\\ \chi\bmod q}}\prod_{\gamma_{\chi}>0}J_0\left(\frac{2\left|\sum_{j=1}^r\chi(a_j)t_j\right|}
{\sqrt{\frac14+\gamma_{\chi}^2}}\right).
\end{equation}
First, we record an exponentially decreasing upper bound for $\hat{\mu}_{q;a_1,\dots,a_r}(\vt)$, which is established in \cite{La}.
\begin{lem}[Proposition 3.2 of \cite{La}] Assume GRH and LI. Fix an integer $r\geq 2$. Let $q$ be large and $0<\epsilon<1/2$ be a real number. Then, uniformly for all $r$-tuples of distinct reduced residues $(a_1,\dots,a_r)$ modulo $q$ we have
$$ |\hat{\mu}_{q;a_1,\dots, a_r}(t_1,\dots,t_r)|\leq \exp(-c_{11}(r)\phi(q)||\vt||),$$ for $||\vt||\geq 400$ and
$$ |\hat{\mu}_{q;a_1,\dots, a_r}(t_1,\dots,t_r)|\leq \exp(-c_{12}(r)\epsilon^2\phi(q)\log q)$$ for  $\epsilon\leq ||\vt||\leq 400$, where $c_{11}(r)$ and $c_{12}(r)$ are positive numbers that depend only on $r$.
\end{lem}
Following the method developed by the author in \cite{La}, we shall derive an asymptotic formula for $\mu_{q;a_1,\dots,a_r}\left(||\vx||>\lambda\sqrt{\Va}\right)$ in the range $0< \lambda\leq \sqrt{\log\log q}$, from which Theorem 1 will be deduced.
\begin{thm}Assume GRH and LI. Fix an integer $r\geq 2$. Let $q$ be large and $a_1,\dots,a_r$ be distinct reduced residue classes modulo $q$. Then in the range $0< \lambda\leq \sqrt{\log\log q}$ we have
\begin{equation*}
\begin{aligned}
 \mu_{q;a_1,\dots,a_r}\left(||\vx||>\lambda\sqrt{\Va}\right)&=
(2\pi)^{-r/2}\int_{||\vx||>\lambda} \exp\left(-\frac{x_1^2+\cdots+x_r^2}{2}\right)d\vx\\
& +\frac{1}{2\Va^2}\sum_{1\leq j<k\leq r}B_q(a_j,a_k)^2 F_{j,k}(\lambda)
+ O_r\left(\frac{(\log\log q)^r}{\log^3q}\right),
\end{aligned}
\end{equation*}
where
$$F_{j,k}(\lambda)= (2\pi)^{-r/2}\int_{||\vx||>\lambda}(x_j^2-1)(x_k^2-1) \exp\left(-\frac{x_1^2+\cdots+x_r^2}{2}\right)d\vx.$$
\end{thm}
\begin{proof}[Proof of Theorem 1]
First, it follows from Corollary 5.4 of \cite{La} that
$$ \max_{(a,b)\in \mathcal{A}(q)}|B_q(a,b)|\ll \phi(q).$$
On the other hand, Proposition 5.1 of \cite{La} yields
$$ |B_q(a,-a)|\gg \phi(q),$$
for all $(a,q)=1$. Hence, we deduce
\begin{equation}
\max_{(a,b)\in \mathcal{A}(q)}\frac{|B_q(a,b)|}{\Va}\asymp \frac{1}{\log q}.
\end{equation}
We also remark that this last estimate follows implicitly from the work of D. Fiorilli and G. Martin \cite{FiM}.

The first part of Theorem 1 follows from Theorem 3.2 upon using (3.2) and noting that
$F_{j,k}(\lambda)\ll_r 1$. Let $1\leq j<k\leq r$. If $1/4<\lambda<3/4$, then
$$
(2\pi)^{-r/2}\int_{||\vx||\leq \lambda}(x_j^2-1)(x_k^2-1)\exp\left(-\frac{x_1^2+\cdots+x_r^2}{2}\right)d\vx \geq \delta_{r},
$$
for some positive number $\delta_r$ which depends only on $r$.
Moreover, we have
$$ (2\pi)^{-r/2}\int_{\vx \in \mathbb{R}^r}(x_j^2-1)(x_k^2-1) \exp\left(-\frac{x_1^2+\cdots+x_r^2}{2}\right)d\vx=0,$$
since the antiderivative of $(x^2-1)e^{-x^2/2}$ is $-xe^{-x^2/2}.$ Hence we deduce, in the range  $1/4<\lambda<3/4$, that
\begin{equation}
F_{j,k}(\lambda)\leq -\delta_r, \text{ for all } 1\leq j<k\leq r.
\end{equation}
On the other hand, it follows from (3.2) that there exist distinct reduced residues $a_1,a_2$ modulo $q$, such that
\begin{equation}
\frac{B_q(a_1,a_2)^2}{\Va^2}\gg \frac{1}{\log^2q}.
\end{equation}
Let $a_3,\dots,a_r$ be distinct reduced residues modulo $q$ that are different from $a_1$ and $a_2$. Appealing to Theorem 3.2 along with (3.3) we get
\begin{equation*}
\begin{aligned}
\left|\mu_{q;a_1,\dots,a_r}\left(||\vx||>\lambda\sqrt{\Va}\right)-
(2\pi)^{-r/2}\int_{||\vx||>\lambda} \exp\left(-\frac{x_1^2+\cdots+x_r^2}{2}\right)d\vx\right|\\
\geq |F_{1,2}(\lambda)|\frac{B_q(a_1,a_2)^2}{2\Va^2}-\kappa_r\frac{(\log\log q)^r}{(\log q)^3}
\end{aligned}
\end{equation*}
for some $\kappa_r>0$ (which depends only on $r$),
if $q$ is sufficiently large. Combining this inequality with (3.3) and (3.4) completes the proof.
\end{proof}
The first step in the proof of Theorem 3.2 consists in using the explicit formula (3.1) to approximate the Fourier transform $\hat{\mu}_{q;a_1,\dots,a_r}(t_1,\dots, t_r)$ by a multivariate Gaussian in the range $||\vt||\ll \sqrt{\log q/\Va}$. To lighten the notation,  we set $\mu=\mu_{q;a_1,\dots, a_r}$ throughout the remaining part of this section.
\begin{lem} Assume GRH and LI. Fix an integer $r\geq 2$ and a real number $A=A(r)>0$. If  $||\vt||\leq A\sqrt{\log q}$, then
\begin{equation*}
 \begin{aligned}
 &\hat{\mu}\left(\frac{t_1}{\sqrt{\Va}},\dots,\frac{t_r}{\sqrt{\Va}}\right)
 =\exp\left(-\frac{t_1^2+\cdots+t_r^2}{2}\right)
 \bigg(1-\frac{1}{\Va}\sum_{1\leq j<k\leq r}B_q(a_j,a_k)t_jt_k\\
&+ \frac{1}{2\Va^2}\sum_{\substack{1\leq j_1<k_1\leq r\\ 1\leq j_2<k_2\leq r}}B_q(a_{j_1},a_{k_1})B_q(a_{j_2},a_{k_2})t_{j_1}t_{j_2}t_{k_1}t_{k_2} + O_r\left(\frac{||\vt||^6}{(\log q)^3}\right)\bigg).
 \end{aligned}
\end{equation*}
\end{lem}
\begin{proof} Using that $\Va\sim \phi(q)\log q$ and $|C_q(a)|=q^{o(1)}$, we infer from (3.1) that
$$ \log\hat{\mu}\left(\frac{t_1}{\sqrt{\Va}},\dots,\frac{t_r}{\sqrt{\Va}}\right)= \sum_{\substack{\chi\neq \chi_0\\ \chi\bmod q}}\sum_{\gamma_{\chi}>0}\log J_0\left(\frac{2\left|\sum_{j=1}^r\chi(a_j)t_j\right|}
{\sqrt{\frac14+\gamma_{\chi}^2}\sqrt{\Va}}\right) + O_{r}\left(\frac{||\vt||}{\Va^{1/4}}\right).$$
Moreover, since $\log J_0(s)=-s^2/4+O(s^4)$ for $|s|\leq 1$, and
 $\sum_{\substack{\chi\neq \chi_0\\ \chi\bmod  q}}\sum_{\gamma_{\chi}>0}1/(\frac{1}{4}+\gamma_{\chi}^2)^2\leq 2 \Va,$ then
\begin{equation}
\log\hat{\mu}\left(\frac{t_1}{\sqrt{\Va}},\dots,\frac{t_r}{\sqrt{\Va}}\right)= -\frac{1}{\Va}\sum_{\substack{\chi\neq \chi_0\\ \chi\bmod q}} \sum_{\gamma_{\chi}>0}\frac{\left|\sum_{j=1}^r\chi(a_j)t_j\right|^{2}}
{\frac14+\gamma_{\chi}^2}+O_r\left(\frac{||\vt||}{\Va^{1/4}}\right),
\end{equation}
in the range $||\vt||\leq A\sqrt{\log q}$. Furthermore, the main term on the RHS of (3.5) equals
\begin{equation}
\begin{aligned}
&-\frac{1}{\Va}\sum_{\substack{\chi\neq \chi_0\\ \chi\bmod q}}\sum_{\gamma_{\chi}>0}\frac{1}{\frac14+\gamma_{\chi}^2}\sum_{1\leq j,k\leq r}\chi(a_j)\overline{\chi(a_k)}t_jt_k\\
&=-\frac{t_1^2+\dots+t_r^2}{2}
-\frac{1}{\Va}\sum_{1\leq j<k\leq r}B_q(a_j,a_k)t_jt_k.\\
\end{aligned}
\end{equation}
On the other hand, in our range of $\vt$, we have
\begin{equation*}
\begin{aligned}
\exp\left(-\sum_{1\leq j<k\leq r}\frac{B_q(a_j,a_k)}{\Va}t_jt_k\right)=& 1-\frac{1}{\Va}\sum_{1\leq j<k\leq r}B_q(a_j,a_k)t_jt_k\\
&+\frac{1}{2\Va^2}\left(\sum_{1\leq j<k\leq r}B_q(a_j,a_k)t_jt_k\right)^2 +O_r\left(\frac{||\vt||^6}{\log^3q}\right),
\end{aligned}
\end{equation*}
which follows from (3.2). Combining this estimate with (3.5) and (3.6) completes the proof.
\end{proof}Let $\Phi(x)=e^{-x^2/2}$ and denote by $\Phi^{(n)}$ the $n$-th derivative of $\Phi$. Then $\Phi^{(1)}(x)=-xe^{-x^2/2}$, $\Phi^{(2)}(x)=(x^2-1)e^{-x^2/2}$, and more generally we know that
$\Phi^{(n)}(x)= (-1)^nH_n(x)e^{-x^2/2}$ where $H_n$ is the $n$-th Hermite polynomial.
We record the following result, which corresponds to Lemma 4.2 of \cite{La}.
\begin{lem}
Let $n_1,\dots, n_r$ be fixed non-negative integers, and $M$ be a large positive number. Then for any  $(x_1,\dots,x_r)\in \mathbb{R}^r,$ we have
$$ \int_{||\mathbf{t}||<M}e^{i(t_1x_1+\cdots+t_rx_r)}
\prod_{j=1}^rt_j^{n_j}e^{-t_j^2/2}
d\mathbf{t}= (2\pi)^{r/2}\prod_{j=1}^r i^{n_j}H_{n_j}(x_j)e^{-x_j^2/2} +O\left(e^{-M^2/4}\right).$$
\end{lem}
Our last ingredient to the proof of Theorem 3.2 is the following basic lemma
\begin{lem} If $f_1, \dots, f_r$ are real valued functions, such that $f_i$ is odd for some $1\leq i\leq r$ then
$$ \int_{V_1<||x||<V_2}f_1(x_1)\cdots f_r(x_r)\exp\left(-\frac{x_1^2+\cdots+x_r^2}{2}\right)d\vx= 0,$$
for all $V_2>V_1\geq 0$.
\end{lem}
\begin{proof}
Using the change of variables $y_j=x_j$ if $j\neq i$ and $y_i=-x_i$, we deduce that the integral we seek to evaluate equals
$$ -\int_{V_1<||y||<V_2}f_1(y_1)\cdots f_r(y_r)\exp\left(-\frac{y_1^2+\cdots+y_r^2}{2}\right)d\vy,$$
since $f_i$ is odd, which establishes the lemma.
\end{proof}

\begin{proof}[Proof of Theorem 3.2]
Let $R:=\sqrt{\Va}\log\log q.$  First, appealing to Theorem 2 we get
\begin{equation}
\mu\left(||\vy||>\lambda\sqrt{\Va}\right)= \mu\left(\lambda\sqrt{\Va}<||\vy||<R\right)+ O\left(\exp\left(-(\log\log q)^{3/2}\right)\right),
\end{equation}
if $q$ is sufficiently large.
Next, we apply the Fourier inversion formula to the measure $\mu$, which yields
$$\mu\left(\lambda\sqrt{\Va}<||\vy||<R\right)
=(2\pi)^{-r}\int_{\lambda\sqrt{\Va}<||\vy||<R}\int_{\vs\in \mathbb{R}^r}e^{i(s_1y_1+\cdots+s_ry_r)}\hat{\mu}(s_1,\dots, s_r)d\vs d\vy.$$
Let $A= A(r)\geq r$ be a suitably large constant. Then, using Lemma 3.1 with $\epsilon:= A(\Va)^{-1/2}\sqrt{\log q}$ we derive
$$  \int_{\mathbf{s}\in \mathbb{R}^r}e^{i(s_1y_1+\cdots+s_ry_r)}\hat{\mu}(s_1,\dots, s_r)d\mathbf{s}= \int_{||\vs||\leq \epsilon}e^{i(s_1y_1+\cdots+s_ry_r)}\hat{\mu}(s_1,\dots, s_r)d\mathbf{s}+ O\left(\frac{1}{q^{2A}}\right).$$
Collecting the above estimates gives
\begin{equation*}
\begin{aligned}
\mu\left(||\vy||>\lambda\sqrt{\Va}\right)=& (2\pi)^{-r}\int_{\lambda\sqrt{\Va}<||\vy||<R}\int_{||\vs||\leq \epsilon}e^{i(s_1y_1+\cdots+s_ry_r)}\hat{\mu}(s_1,\dots, s_r)d\vs d\vy\\
& +  O\left(\exp\left(-(\log\log q)^{3/2}\right)\right),\\
\end{aligned}
\end{equation*}
using that $R^rq^{-2A}\ll q^{-r}.$ Upon making the change of variables
$$ t_j:=  \sqrt{\Va} s_j, \text{ and } x_j:= \frac{y_j}{\sqrt{\Va}},$$
we infer that $\mu\left(||\vx||>\lambda\sqrt{\Va}\right)$ equals
\begin{equation}
\begin{aligned}
& (2\pi)^{-r}\int_{\lambda<||\vx||<\log\log q}\int_{||\vt||\leq A\sqrt{\log q}}e^{i(t_1x_1+\cdots+t_rx_r)}\hat{\mu}\left(\frac{t_1}{\sqrt{\Va}},\dots, \frac{t_r}{\sqrt{\Va}}\right)d\vt d\vx\\
& +  O\left(\exp\left(-(\log\log q)^{3/2}\right)\right),\\
\end{aligned}
\end{equation}
Now we use the asymptotic expansion of $\hat{\mu}\left(t_1\Va^{-1/2},\dots, t_r\Va^{-1/2}\right)$ proved in Lemma 3.3. First, the contribution of the error term in this asymptotic to the integral in (3.8) is
$$ \ll_r \frac{(\log\log q)^r}{(\log q)^3}.$$
Next we shall compute the contribution of the main terms. Appealing to Lemma 3.4, we obtain
\begin{equation*}
\begin{aligned}
&(2\pi)^{-r}\int_{\lambda<||\vx||<\log\log q}\int_{||\vt||\leq A\sqrt{\log q}}e^{i(t_1x_1+\cdots+t_rx_r)}
\exp\left(-\frac{t_1^2+\cdots+t_r^2}{2}\right)d\vt d\vx\\
&=(2\pi)^{-r/2}\int_{\lambda<||\vx||<\log\log q}\exp\left(-\frac{x_1^2+\cdots+x_r^2}{2}\right) d\mathbf{x}+O\left(\frac{1}{q^A}\right)\\
&=(2\pi)^{-r/2}\int_{||\vx||>\lambda}\exp\left(-\frac{x_1^2+\cdots+x_r^2}{2}
\right)d\mathbf{x}+O\left(\exp\left(-(\log\log q)^{3/2}\right)\right),\\
\end{aligned}
\end{equation*}
if $q$ is large enough.
Similarly, we deduce from Lemma 3.4 that for $1\leq j<k\leq r$ we have
\begin{equation*}
\begin{aligned}
&(2\pi)^{-r}\int_{\lambda<||\vx||<\log\log q}\int_{||\vt||\leq A\sqrt{\log q} }t_jt_ke^{i(t_1x_1+\cdots+t_rx_r)}\exp\left(-\frac{t_1^2+\cdots+t_r^2}{2}\right)
d\vt d\vx\\
&= -(2\pi)^{-r/2}\int_{\lambda<||\vx||<\log\log q}x_jx_k\exp\left(-\frac{x_1^2+\cdots+x_r^2}{2}\right) d\mathbf{x}+O\left(\frac{1}{q^A}\right)= O\left(\frac{1}{q^A}\right),\\
\end{aligned}
\end{equation*}
which follows from Lemma 3.5. Furthermore, if $1\leq j_1<k_1\leq r$ and $1\leq j_2<k_2\leq r$, then a similar argument along with Lemmas 3.4 and 3.5 shows that
$$
(2\pi)^{-r}\int_{\lambda<||\vx||<\log\log q}\int_{||\vt||\leq A\sqrt{\log q} }t_{j_1}t_{k_1}t_{j_2}t_{k_2}e^{i(t_1x_1+\cdots+t_rx_r)}\exp\left(-\frac{t_1^2+\cdots+t_r^2}{2}\right)
d\vt d\vx
$$
is $\ll q^{-A}$ if $j_1\neq j_2$ or $k_1\neq k_2$, and equals
$F_{j,k}(\lambda)+O\left(\exp\left(-(\log\log q)^{3/2}\right)\right),$
if $j_1=j_2=j$ and $k_1=k_2=k.$
The theorem now follows upon collecting the above estimates and using Lemma 3.3.
\end{proof}

\section{Very large deviations of $\mu_{q;a_1,\dots,a_r}$: Proof of Theorem 4}
In this section we prove a precise estimate for $\mu_{q;a_1,\dots,a_r}(|x|_{\infty}>V)$ in the range $V/(\phi(q)\log^2q)\to\infty$ as $q\to\infty.$ The advantage of using the maximum norm is that the bounds in Lemma 2.1 can be made sharp, so that a precise estimate for large deviations of $\mu_{q;a_1,\dots,a_r}$ would follow from a close investigation of the tails of $Y(q)$. Indeed, we have
\begin{lem} Let $\epsilon>0$ be small. If $q$ is sufficiently large, $2\leq r\leq \phi(q)-1$ and $V\geq \sqrt{\phi(q)}$, then for all distinct reduced residues $a_1,\dots,a_r$ modulo $q$, we have
$$
 \pr\left(Y(q)> V+q^{\epsilon}\right)\leq \mu_{q;a_1,\dots,a_r}(|x|_{\infty}>V)\leq 2r\pr\left(Y(q)> V-q^{\epsilon}\right).
$$
\end{lem}
\begin{proof} The result can be derived along the same lines as in the proof of Lemma 2.1 upon noting that
$$\mu_{q;a_1,\dots,a_r}(|x|_{\infty}>V)\leq \sum_{j=1}^r\pr\left(|X(q,a_j)|> V\right)\leq \sum_{j=1}^r\pr\left(|Y(q)|> V-|C_q(a_j)|\right),$$
and $|C_q(a)|=q^{o(1)}$ for all $(a,q)=1$.
\end{proof}
In view of this lemma, it suffices to obtain the analogous estimate for
$$ \rho_q(V):=\pr(Y(q)> V).$$
Since $\ex(\exp(t\cos(2\pi\theta(\gamma_{\chi}))))=I_0(t)$, then
\begin{equation}\mathcal{L}(s):=\int_{-\infty}^{\infty}se^{sV} \rho_q(V)dV=\prod_{\substack{\chi\neq \chi_0\\ \chi\bmod q}}\prod_{\gamma_{\chi}>0}I_0
\left(\frac{2s}{\sqrt{\frac14+\gamma_{\chi}^2}}\right), \text{ for all } s>0.
\end{equation}
To prove Theorem 4, we shall establish an asymptotic formula for  $\log \mathcal{L}(s)$ in a wide range of $s$, and then use the saddle-point method to extract the desired estimate for $\rho_{q}(V)$ from that of $\mathcal{L}(s)$. We first collect some useful estimates and properties of the Bessel function $I_0(t)$.
\begin{lem}
 $\log I_0(t)$ is a smooth function with bounded derivative on $[0,+\infty)$
and satisfies
$$
\log I_0(t)=\begin{cases} O\left(t^2\right) & \text{ if }  0\leq t<1\\  t+O\left(\log3t\right) & \text{ if }  t\geq 1. \end{cases}
$$
\end{lem}
\begin{proof}
The first estimate follows from the Taylor series
$I_0(t)= \sum_{n=0}^{\infty}
(t/2)^{2n}/n!^2.$
Moreover, we have
\begin{equation}
I_0(t)=\int_{0}^{1}e^{t\cos(2\pi\theta)}d\theta\leq e^t.
\end{equation} On the other hand, taking $\epsilon=\frac{1}{2\pi t}$ we deduce
$$ I_0(t)\geq \int_{0}^{\epsilon}e^{t\cos(2\pi\theta)}d\theta\geq \epsilon e^{t\cos(2\pi\epsilon)}\geq \frac{e^t}{10\pi t}.$$
This together with (4.2) yields the second estimate.
Finally, since $I_0(t)$ is a positive smooth function on $[0,+\infty)$ then $\log I_0(t)$ is smooth and
$$ |(\log I_0(t))'|= \left|\frac{\int_{0}^{1}\cos(2\pi\theta) e^{t\cos(2\pi\theta)}d\theta}{\int_{0}^{1} e^{t\cos(2\pi\theta)}d\theta}\right|\leq 1.$$
\end{proof}
 Let $f(t)$ be the real valued function defined by
$$ f(t)=\begin{cases} \log I_0(t) & \text{ if } 0\leq t<1\\ \log I_0(t)-t &\text{ if } t\geq 1.
\end{cases}
$$ We prove
\begin{pro} If $\log s/\log q\to \infty$ as $q\to \infty$, then
 $$\mathcal{L}(s)=\exp\left(\frac{\phi(q)-1}{2\pi}s\log^2 s+ \frac{D(q)}{\pi}s\log s+O(s\phi(q)\log q)\right),$$ where
$$ D(q)=\phi(q)\left(\log q-\sum_{p|q}\frac{\log p}{p-1}\right)+(\phi(q)-1)\left(\int_{0}^{\infty}\frac{f(t)}{t^2}dt-\log \pi\right).$$
\end{pro}

\begin{proof}
First, it follows from Lemma 4.2 that $\log I_0$ is a Lipshitz function. Therefore, using (4.1) and Lemma 2.4 we get
\begin{equation}
\log \mathcal{L}(s)=\sum_{\substack{\chi\neq \chi_0\\ \chi\bmod q}}\sum_{\gamma_{\chi}>0}\log I_0
\left(\frac{2s}{\sqrt{\frac14+\gamma_{\chi}^2}}\right)=\sum_{\substack{\chi\neq \chi_0\\ \chi\bmod q}}\sum_{\gamma_{\chi}>2}\log I_0
\left(\frac{2s}{\gamma_{\chi}}\right)+E_2,
\end{equation}
where
\begin{equation*}
\begin{aligned}
 E_2&\ll s\sum_{\substack{\chi\neq \chi_0\\ \chi\bmod q}}\sum_{0<\gamma_{\chi}<2}\frac{1}{\sqrt{\frac14+ \gamma_{\chi}^2}}+s\sum_{\substack{\chi\neq \chi_0\\ \chi\bmod q}}\sum_{\gamma_{\chi}>2}
\left(\frac{1}{\gamma_{\chi}}-\frac{1}{\sqrt{\frac14+\gamma_{\chi}^2}}\right)\\
&\ll sN_q(2)+ s\sum_{\substack{\chi\neq \chi_0\\ \chi\bmod q}}\sum_{\gamma_{\chi}>2}\frac{1}{\frac14+\gamma_{\chi}^2}\ll s\phi(q)\log q.
\end{aligned}
\end{equation*}
 Furthermore, note that
\begin{equation} \sum_{\substack{\chi\neq \chi_0\\ \chi\bmod q}}\sum_{\gamma_{\chi}>2}\log I_0
\left(\frac{2s}{\gamma_{\chi}}\right)=2s\sum_{\substack{\chi\neq \chi_0\\ \chi\bmod q}}\sum_{2<\gamma_{\chi}\leq 2s}
\frac{1}{\gamma_{\chi}}+\sum_{\substack{\chi\neq \chi_0\\ \chi\bmod q}}\sum_{\gamma_{\chi}>2}f
\left(\frac{2s}{\gamma_{\chi}}\right).
\end{equation}
Now, we infer from Lemma 2.4 that
\begin{equation*}
\begin{aligned}
&\sum_{\substack{\chi\neq \chi_0\\ \chi\bmod q}}\sum_{2<\gamma_{\chi}\leq 2s}
\frac{1}{\gamma_{\chi}}=\int_{2}^{2s}\frac{d N_q(t)}{t}= \frac{N_q(2s)}{2s}+\int_{2}^{2s}\frac{N_q(t)}{t^2}dt+O(\phi(q)\log q)\\
&=\frac{(\phi(q)-1)}{2\pi}\log s+\frac{1}{2\pi}\int_{2}^{2s}\left((\phi(q)-1)\frac{\log t}{t}+\frac{R(q)}{t}\right)dt+O(\phi(q)\log q)\\
&= \frac{(\phi(q)-1)}{4\pi}\log^2s +\frac{1}{2\pi}(R(q)+(\phi(q)-1)(1+\log 2))\log s +O(\phi(q)\log q).\\
\end{aligned}
\end{equation*}
On the other hand, the second sum on the RHS of (4.4) equals
\begin{equation}
\int_{2}^{\infty}f\left(\frac{2s}{t}\right)d N_q(t)
= -\int_2^{\infty}
f\left(\frac{2s}{t}\right)'\left(\frac{\phi(q)-1}{2\pi}t\log t+\frac{R(q)}{2\pi}t\right)dt + E_3,
\end{equation}
where
$$ E_3\ll |f(s)N_q(2)| + \phi(q)\int_{2}^{\infty}\frac{s}{t^2}
\left|f'\left(\frac{2s}{t}\right)\right|
\log(qt) dt \ll s\phi(q)\log q, $$
which follows from Lemmas 2.4 and 4.2 along with the fact that $\lim_{t\to\infty}f\left(\frac{2s}{t}\right)N_q(t)=0$.
This yields
$$ \int_{2}^{\infty}f\left(\frac{2s}{t}\right)dN_q(t)=\int_2^{\infty}
f\left(\frac{2s}{t}\right)\left(\frac{\phi(q)-1}{2\pi}\log t+\frac{R(q)+\phi(q)-1}{2\pi}\right)dt+ O(s\phi(q)\log q).$$
Making the change of variable
$u=2s/t$, we deduce that the integral on the RHS of the previous estimate equals
\begin{align*}
&2s\int_0^{s}\frac{f(u)}{u^2}\left(\frac{\phi(q)-1}{2\pi}\log \left(\frac{2s}{u}\right)+\frac{R(q)+\phi(q)-1}{2\pi}\right)du\\
&= \frac{ (\phi(q)-1)}{\pi}s\log s \int_0^{s}\frac{f(u)}{u^2}du+O(s\phi(q)\log q)\\
&=\frac{ (\phi(q)-1)}{\pi}s\log s \int_0^{\infty}\frac{f(u)}{u^2}du+O(s\phi(q)\log q),
\end{align*}
using that $R(q)\sim \phi(q)\log q$ by Lemma 2.3 and
$$ \int_0^{\infty}\frac{f(u)\log u}{u^2}du<\infty \ \text{ and } \ \int_{s}^{\infty}\frac{f(u)}{u^2}du\ll \int_{s}^{\infty}\frac{\log u}{u^2}du\ll \frac{\log s}{s},$$
which follow from Lemma 4.2.
The proposition follows upon collecting the above estimates.
\end{proof}

\begin{proof}[Proof of Theorem 4]
Assume that $V$ is such that $V/(\phi(q)\log^2q)\to\infty$ as $q\to\infty$, and let $s=s(V)>1$ be the unique solution to the equation
\begin{equation}
2\pi V= (\phi(q)-1)\log^2 s+2\left(D(q)+\phi(q)-1\right)\log s.
\end{equation}
Then
\begin{equation}
 s= \exp\left(-\frac{D(q)}{\phi(q)-1}-1 +\sqrt{\left(\frac{D(q)}{\phi(q)-1}+1\right)^2+\frac{2\pi V}{\phi(q)-1}}\right).
\end{equation}
Let $\epsilon, \delta>0$ be small real numbers to be chosen later and define
$$ s_1:=(1+\epsilon)s, s_2=(1-\epsilon)s, \text{ and } V_1:=(1+\delta)V, V_2:=(1-\delta)V.$$
Moreover, let
$$ I_1 :=\int_{V_1}^{\infty}se^{st} \rho_q(t)dt, \ \text{ and } \ I_2:=\int_{-\infty}^{V_2}se^{st} \rho_q(t)dt.$$
Then
\begin{equation}
\frac{I_1}{\mathcal{L}(s)}= \frac{1}{(1+\epsilon)\mathcal{L}(s)}\int_{V_1}^{\infty} e^{-\epsilon st}s_1e^{s_1 t}\rho_q(t)dt
\leq e^{-\epsilon sV_1}\frac{\mathcal{L}(s_1)}{(1+\epsilon)\mathcal{L}(s)}.
\end{equation}
Therefore, appealing to Proposition 4.3 we obtain that the RHS of the previous inequality equals
$$
\exp\left(-\epsilon sV_1+\frac{\phi(q)-1}{2\pi}(s_1\log^2 s_1-s\log^2 s)+ \frac{D(q)}{\pi}(s_1\log s_1-s\log s) +O(s\phi(q)\log q)\right),$$
which in view of (4.6) becomes
$$ \exp\left(-\epsilon\delta sV+\frac{\phi(q)-1}{\pi}s\log s((1+\epsilon)\log(1+\epsilon)-\epsilon)+O(s\phi(q)\log q)\right).$$
Combining this estimate with (4.8) and using that $(1+\epsilon)\log(1+\epsilon)-\epsilon\leq \epsilon^2$, we deduce
$$
\frac{I_1}{\mathcal{L}(s)}\leq \exp\left(-\epsilon\delta sV+\epsilon^2\frac{\phi(q)-1}{\pi}s\log s +O(s\phi(q)\log q)\right).
$$
Now, (4.6) yields $2\pi V\geq (\phi(q)-1)\log^2 s$. Therefore, choosing $\delta= 4\epsilon/\log s$ we get
$$
\frac{I_1}{\mathcal{L}(s)}\leq \exp\left(-\frac{\epsilon^2}{2\pi}(\phi(q)-1) s\log s +O(s\phi(q)\log q)\right).$$
Furthermore, taking $\epsilon=C\sqrt{\log q/\log s}$ for a suitably large constant $C>0$, we obtain
$$
\frac{I_1}{\mathcal{L}(s)}\leq \exp\left(-s\phi(q)\log q\right).
$$
A similar argument shows that
$$
\frac{I_2}{\mathcal{L}(s)}\leq  \exp\left(-s\phi(q)\log q\right).
$$
Combining these two inequalities with Proposition 4.3 we deduce
$$
\int_{V(1-\delta)}^{V(1+\delta)}s e^{st}\rho_q(t)dt=\exp\left(\frac{\phi(q)-1}{2\pi}s\log^2 s+ \frac{D(q)}{\pi}s\log s+O(s\phi(q)\log q)\right).
$$
On the other hand, since $V\ll \phi(q)\log ^2s$ then
$$\int_{V(1-\delta)}^{V(1+\delta)}se^{st}dt=\exp\left(sV+ O\left(s\phi(q)\sqrt{\log s\log q}\right)\right).$$
Hence, using that $\rho_q$ is a non-increasing function we infer from the two previous estimates that
\begin{equation}
 \rho_q(V(1+\delta))\leq \exp\left(-\frac{\phi(q)-1}{\pi}s\log s+ O\left(s\phi(q)\sqrt{\log s\log q}\right)\right) \leq \rho_q(V(1-\delta)).
\end{equation}
Now, using that $D(q)\sim \phi(q)\log q$ by Lemma 2.3 we get
\begin{equation}
 \log s= \sqrt{\frac{2\pi V}{\phi(q)-1}}\left(1+ O\left(\left(\frac{\phi(q)\log^2 q}{V}\right)^{1/2}\right)\right),
\end{equation}
and
\begin{equation}
\sqrt{\frac{\log q}{\log s}}\ll \left(\frac{\phi(q)\log^2q}{V}\right)^{1/4}.
\end{equation}
Moreover, we have
\begin{equation}
\begin{aligned}
&\exp\left(-\frac{D(q)}{\phi(q)-1}-1 +\sqrt{\left(\frac{D(q)}{\phi(q)-1}+1\right)^2+\frac{2\pi V(1+O(\delta))}{\phi(q)-1}}\right)\\
&=\exp\left(-\frac{D(q)}{\phi(q)-1}-1 +\sqrt{\left(\frac{D(q)}{\phi(q)-1}+1\right)^2+\frac{2\pi V}{\phi(q)-1}}\right)\left(1+ O\left(\left(\frac{\phi(q)\log^2q}{V}\right)^{1/4}\right)\right).
\end{aligned}
\end{equation}
Finally, the theorem follows upon using (4.6) and inserting the estimates (4.10)-(4.12) into (4.9).

\end{proof}

\end{document}